\documentclass[12pt]{amsart}

\usepackage{amssymb,amsbsy}
\usepackage{latexsym,array}
\usepackage{verbatim,color}
\usepackage{fullpage,euscript}
\usepackage{xcolor}
\usepackage{tikz}
\usetikzlibrary{arrows,shapes,trees}
\usepackage[colorlinks=true,linkcolor=blue,urlcolor=my_color,citecolor=magenta]{hyperref}

\definecolor{my_color}{rgb}{0,0.5,0.5}
\definecolor{MIXT}{rgb}{0.8,0.5,0.2}
\definecolor{mixt}{rgb}{0.5,0.3,0.2}
\definecolor{sin}{rgb}{0,0.5,0.5}
\definecolor{darkblue}{rgb}{0,0.1,0.8}
\definecolor{redi}{rgb}{0.5,0,0.4}

\numberwithin{equation}{section}

\makeatletter
\@namedef{subjclassname@2020}{\textup{2020} Mathematics Subject Classification}
\makeatother

\newtheorem{thm}{Theorem}[section]
\newtheorem{lm}[thm]{Lemma}
\newtheorem{cl}[thm]{Corollary}
\newtheorem{prop}[thm]{Proposition}

\theoremstyle{remark}
\newtheorem{rmk}[thm]{Remark}

\theoremstyle{definition}
\newtheorem{ex}[thm]{Example} 
\newtheorem{df}[thm]{Definition}

\newcommand {\g}{{\mathfrak g}}

\newcommand {\q}{{\mathfrak q}}

\newcommand{\gt}{\mathfrak}

\newcommand {\eus}{\EuScript}

\newcommand {\gS}{{\eus S}}
\newcommand {\U}{{\eus U}}

\newcommand {\ap}{\alpha}

\newcommand {\C}{{\mathbb C}}

\newcommand {\BZ}{{\mathbb Z}}

\newcommand {\R}{{\mathbb R}}

\newcommand{\id}{{\mathsf{id}}}
\newcommand {\ad}{{\mathrm{ad}}}

\newcommand {\Ad}{{\mathrm{Ad\,}}}

\newcommand {\Ann}{{\mathrm{Ann}}}

\newcommand {\ind}{{\mathrm{ind\,}}}
\newcommand {\Lie}{{\mathsf{Lie\,}}}

\newcommand {\rk}{{\mathsf{rk\,}}}

\newcommand {\trdeg}{{\mathrm{tr.deg\,}}}
\newcommand {\diag}{{\mathrm{diag}}}

\newcommand{\Quot}{{\mathrm{Quot}\,}}
\newcommand{\SL}{\mathrm{SL}}
\newcommand{\GL}{\mathrm{GL}}

\newcommand{\wb}{\widehat{\gt b}}

\newcommand {\bp}{{\boldsymbol{p}}}

\newcommand {\fb}{\boldsymbol{f}}
\newcommand{\bwedge}{\text{\scalebox{1.3}{$\wedge$}}}
\newcommand{\ard}{\rightsquigarrow}

\newcommand {\cW}{{\mathcal W}}

\newcommand {\beq}{\begin{equation}}
\newcommand {\eeq}{\end{equation}}

\renewcommand{\le}{\leqslant}
\renewcommand{\ge}{\geqslant}

\newcommand {\bbk}{\Bbbk}

\begin{document}
\setlength{\parskip}{3pt plus 2pt minus 0pt}
\hfill { {\color{blue}\scriptsize  January 14, 2025} 
\vskip1ex

\title[On contact Lie algebras]
{Contact Lie algebras, generic stabilisers, and affine seaweeds}
\author[O.\,Yakimova]{Oksana S.~Yakimova}
\address[O.\,Yakimova]{Institut f\"ur Mathematik, Friedrich-Schiller-Universit\"at Jena, Jena, 07737, Deutschland}
\email{oksana.yakimova@uni-jena.de}
\thanks{This work is funded by the Deutsche Forschungsgemeinschaft (DFG, German Research Foundation) --- project number 404144169.}
\begin{abstract}
Let $\q=\Lie Q$ be an algebraic  Lie algebra of index $1$, i.e., a generic $Q$-orbit on $\q^*$ has codimension $1$. We show that the following conditions are equivalent: 
$\q$ is contact; a generic $Q$-orbit on $\q^*$ is not conical; there is a generic stabiliser for the coadjoint action of $\q$. 
In addition, if $\q$ is contact, then the subalgebra $\gS(\q)_{\sf si}\subset\gS(\q)$ generated by 
symmetric semi-invariants of $\q$ is a polynomial ring. 
We study also affine seaweed Lie algebras of type {\sf A} and find some contact as well as non-contact examples among them.  
\end{abstract}
\maketitle

\section*{Introduction}

\noindent
Let $\q$ be a 
   Lie algebra over a
 field $\bbk$ of characteristic zero. 
For $\alpha\in\gt q^*$, let $\textsl{d}_{\gt q}\alpha\in\bwedge^2\gt q^*$ be the image of $\alpha$ under the  
Chevalley--Eilenberg differential. Suppose that $\dim\gt q=2n+1$. 
Then $\gt q$ is said to be {\it contact}, if there is $\alpha\in\gt q^*$ such that 
$(\bwedge^n \textsl{d}_{\gt q}\alpha)\wedge\alpha\ne 0$. This definition originated from geometric constructions. 

A contact structure on a smooth manifold $M$ of dimension $2n+1$ is a differential 1-form $v$ such that 
$(dv)^n \wedge v \ne 0$ at each point of $M$, here $dv$ is the de Rham differential of $v$. 
For information about  contact geometry or topology, see e.g. \cite{gei}. 
According to Gromov \cite{gro}, there is a contact structure on every odd-dimensional connected non-compact real Lie group $Q$. 
In general, such contact structures are not invariant under left  
translations by the group elements. Furthermore, 
$Q$ admits an invariant contact form if and only if its Lie algebra $\Lie Q$ is contact, see e.g. \cite[Sect.\,2]{dia}.
 
Let $\gt q_\alpha\subset\gt q$ be the kernel of the skew-symmetric bilinear form 
 $\textsl{d}_{\gt q}\alpha$. Then $\gt q_\alpha$ is also the stabiliser of $\ap$ w.r.t. the coadjoint action. 
The {\it index of $\q$}  is defined by 
\begin{equation}\label{ind}
\ind\gt q=\min_{\gamma\in\gt q^*} \dim\gt q_\gamma.
\end{equation}  
 If $\alpha\in\q^*$ is a contact linear function, i.e., 
 $(\bwedge^n \textsl{d}_{\gt q}\alpha)\wedge\alpha\ne 0$, then $\dim\gt q_\alpha=1$.  
Thereby each contact Lie algebra is of index 1. However, not any $\gt q$ with $\ind\gt q=1$ is contact. Nevertheless, there are classes of Lie algebras, where these two properties are equivalent.
For instance, this is true  for algebraic Lie algebras whose  
radicals consist of $\ad$-nilpotent elements, see Corollary~\ref{no}.   This class includes all 
nilpotent Lie algebras, see also Example~\ref{nilp} for special feature of the nilpotent case. 

Until the end of the Introduction assume that $\bbk$ is algebraically closed and that 
$\gt q$ is an algebraic Lie algebra, i.e., $\gt q=\Lie Q$, where $Q$ is a connected affine algebraic group.  
In Section~\ref{sec1},  we observe first that  
$\gt q$ with $\ind\gt q=1$ is contact if and only if a generic $Q$-orbit in $\gt q^*$ is not conical or, equivalently, 
if the extended group $\widetilde{Q}=Q\times\bbk^{\!^\times}$ acts on $\gt q^*$ with an open orbit. 
This implies that 
a contact form is unique up to conjugation by elements of $Q$ and scalar multiplication. 

A linear form $\gamma\in\gt q^*$ is {\it stable} in the terminology of \cite{TYu} and its stabiliser $\gt q_\gamma\subset\gt q$ 
is a {\it generic stabiliser for the coadjoint action}, if there is a non-empty open subset $U\subset \gt q^*$ such that  
$\gt q_\gamma$ and $\gt q_\beta$ are conjugate by an element of $Q$ for each $\beta\in U$. 

Our main result, Theorem~\ref{equiv2}, states that 
a Lie algebra $\gt q$ of index 1 is contact if and only if  there is a generic stabiliser for the coadjoint action.
In particular, if $\alpha\in\gt q^*$ is contact, then it is stable.  

In Section~\ref{sec-products},
we deal with semi-direct products 
$Q=L\ltimes\exp(V)$, where $\exp(V)$ is a normal Abelian unipotent subgroup, $L$ acts on 
$V^*$ with an open orbit $L\gamma$, and $\ind\gt q=1$. 
We explain how to check whether $\gt q$ is contact or not. The answer is given in terms of the 
stabiliser $\gt l_\gamma=\Lie L_\gamma$ 
 and the normaliser $\gt l_{<\gamma>}$ of the line $\bbk\gamma$. 
For instance, if $\gt l_\gamma$ is not contact and $\ind\gt l_{<\gamma>}=0$, then $\gt q$ is contact,
see Theorem~\ref{impl}\,{\sf (i)}. 
If $\gt l_\gamma$ is contact and $\ind\gt l_{<\gamma>}=0$, then $\gt q$ may be contact or not. This depend on 
the eigenvalue on $\gamma$ of 
a special semisimple element $s\in\gt l_{<\gamma>}$, see Proposition~\ref{impl-ii}.

In Section~\ref{sec-inv}, we consider the subring  $\gS(\q)^\q$ of symmetric invariants and the ring $\gS(\q)_{\sf si}\subset\gS(\q)$ 
generated by 
semi-invariants of $\q$.
If $\q$ is contact, then  $\gS(\q)_{\sf si}$  is a finitely generated 
polynomial ring, 
see  Proposition~\ref{free}.  
Being contact is essential for our conclusion. There are non-contact Lie algebras of index 1 such that $\gS(\gt q)_{\sf si}$ is not a polynomial ring, see 
Example~\ref{not-free}. 
Our proof relies on an old result of Sato--Kimura \cite{SK}. Their method applies also to Lie algebras of index zero and leads to a similar conclusion, see \cite[Sect.\,3.2]{kot-T}. A different   approach to  $\gS(\q)_{\sf si}$ in case 
$\ind\q=0$ is developed in   \cite{Ooms}. 

Let $\q_{\rm tr}\subset\q$ be {\it the canonical truncation} of a Lie algebra $\gt q$ considered in \cite{bgr,fonya-et,kot-T}. 
If $\q$ is contact, then 
$\gS(\q_{\rm tr})^{\q_{\rm tr}}$ is a polynomial ring in $\trdeg \gS(\q)_{\sf si}=\ind\q_{\rm tr}$  variables; 
furthermore,  this property extends to  each finite-dimensional quotient 
$\gt q_{\rm tr}[t]/(t^k)$ of the current algebra $\gt q_{\rm tr}[t]$, see  Section~\ref{sec-tr}. Non-reductive Lie algebras $\gt s$ such that $\gS(\gt s)^{\gt s}$ is a polynomial ring with $\ind\gt s$ generators attract a lot of attention, see e.g.  
\cite{takiff,jos,p09,MZ,kot-T}.  A quest for this type of algebras continues. Our results provide another source of them. 

The ring $\gS(\q)^\q$ itself is less spectacular. 
If $\ind\q=1$, then either
$\gS(\q)^\q=\bbk$ or $\gS(\q)^\q$ is generated by one homogeneous polynomial,
see Section~\ref{gen}.  If $\ind\q=1$ and $\gS(\q)^\q\ne\bbk$, then $\q$ is contact by
Proposition~\ref{inv1}. 

Section~\ref{sec-aff} is devoted to {\it affine seaweeds}. These Lie algebras are analogues of the usual seaweeds in the setting of loop algebras, see Section~\ref{expl} for details on seaweeds in reductive Lie algebras. For one particular series of 
affine seaweeds in type {\sf A}$_r$, we derive an explicit formula for the index. 

Let 
$\gt p\subset\gt{sl}_{r+1}$ be a maximal parabolic with the diagonal blocks of sizes $a$ and $b$. Write 
$\gt p=\gt l\ltimes\gt n$, where $\gt n\cong\bbk^a{\otimes}(\bbk^b)^*$ is the nilpotent radical, which is Abelian. 
We add another copy of $\gt n$ obtaining $\gt q=\bar{\gt q}(a,b)=\gt l\ltimes (\gt n\oplus\gt n)$. Then $\q$ is an affine seaweed and 
$\ind\q=\gcd(2a,a+b)-1$ by Theorem~\ref{t-ind}. If $\gcd(2a,a+b)=2$, then $\ind\bar\q(a,b)=1$.
A natural question is whether this Lie algebra is contact or not. As it turns out, the question is rather difficult. 
If $a$ and $b$ are even and $\ind\bar\q(a,b)=1$, then $\bar\q(a,b)$ is quasi-reductive and contact,  see Theorem~\ref{sw1}. On the contrary, $\bar\q(1,b)$ is not contact for any odd $b$, see Example~\ref{notc}.  
There is an ample opportunity for further investigation of the coadjoint action of an affine seaweed. 

In \cite{coll},  it is shown that a seaweed subalgebra of $\gt{sl}_r(\C)$ or $\gt{sp}_{2r}(\C)$ that has index $1$ is contact.  
To be more precise, the main theorem of \cite{coll} states that 
an index-one seaweed of a complex simple Lie algebra is contact if and only if it is quasi-reductive.
By a result of Panyushev \cite{Dima03-b} each seaweed in $\gt{sl}_r$ or $\gt{sp}_{2r}$ is quasi-reductive.
 The equivalence of  Theorem~\ref{equiv2} explains, simplifies, and generalises  
arguments of \cite{coll}, see Section~\ref{expl}. 

The authors of \cite{coll} claim that their result provides a classification of contact seaweeds. Unfortunately,
this is very far from the truth. In spite of many formulas for the index of a seaweed \cite{derK,Dima,jos,mC},  
no one knows how to list all seaweeds of index 1 in  $\gt{sl}_r$ or in $\gt{sp}_{2r}$.  
Our results in Section~\ref{sec-products} indicate that a classification  of the contact Lie algebras is hardly possible.

\section{Preliminaries and notation} \label{sec-prem}

For an irreducible affine variety $Y$ over $\bbk$, we let $\bbk[Y]$ be the  ring of regular functions on $Y$ and 
$\bbk(Y)=\Quot \bbk[Y]$  the field of rational functions on $Y$.    
A statement that a certain assertion  holds for {\it generic points} of $Y$ (or for generic orbits on $Y$) means that this assertion holds  for all points of a non-empty open subset $U\subset Y$ (for all orbits intersecting $U$). If an algebraic group $Q$ acts on $Y$, then 
$\bbk[Y]^Q$  is the ring of $Q$-invariant regular functions   
and $\bbk(Y)^Q$ is the field of $Q$-invariant rational functions on $Y$. 

Suppose $\alpha\in\gt q^*$.
The kernel $\gt q_\alpha\subset\gt q$ of the skew-symmetric form $\textsl{d}_{\gt q}\alpha$ is 
defined by 
\[
\gt q_\alpha=\{\xi\in\gt q\mid \alpha([\xi,\gt q])=0\}.
\]
It is the stabiliser of $\alpha$ in $\gt q$. Let $Q_\alpha\subset Q$ be the stabiliser of $\alpha$ for the 
coadjoint action. Then 
$\gt q_\alpha = \Lie Q_\alpha$. Therefore $\dim Q\alpha=\dim\gt q-\dim\gt q_\alpha$. Hence $\ind\gt q$  is 
the minimal codimension of a $Q$-orbit in $\q^*$, see~\eqref{ind} for the definition of index.  If $Q$ is an algebraic group, then 
\begin{equation} \label{tr-ind}
\ind\gt q=\trdeg\bbk(\q^*)^Q
\end{equation} 
by the Rosenlicht theorem, see~\cite[IV.2]{spr}. 

Set $\gt q^*_{\sf reg}=\{\gamma\in\gt q^*\mid \dim\gt q_\gamma=\ind\gt q\}$ and 
$\gt q^*_{\sf sing}=\gt q^*\setminus \gt q^*_{\sf reg}$. 
We say that $\q$ has the {\sl codim}--$2$ property if $\dim \q^*_{\sf sing}\le \dim\q-2$. 

Let $V$ be a finite-dimensional vector space over $\bbk$ and $V^*$ the dual space. 
Then $(V^*)^*$ is canonically isomorphic to $V$. For a subspace $W\subset V$, 
let $\Ann(W)\subset V^*$ be the {\it annihilator} of $W$.  

Over an algebraically closed field, an orbit $Qy$ of a group $Q$ is said to be {\it conical}, if 
$\bbk^{\!^\times}\! y\subset Qy$. For the coadjoint representation 
\[
\ad^*\!: \q \to\gt{gl}(\q^*)
\]
and $\ap\in\gt q^*$, 
an equivalent condition is that $\ap\in\ad^*(\q){\cdot}\ap$. 

\begin{lm}[{cf. \cite[Lemma\,2.8]{MRS}}] \label{alg-con}
For any finite-dimensional Lie algebra $\gt q$ and any $\alpha\in\gt q^*$, there is an equivalence:\,  
$\ap\in\ad^*(\q){\cdot}\ap\ \Leftrightarrow \ \alpha(\q_\alpha)=0$. 
\end{lm}
\begin{proof}
If $\alpha\in\ad^*(\gt q){\cdot}\alpha$, then $\alpha=\ad^*(\xi)(\alpha)$ 
for some $\xi\in\gt q$ and 
\[
\alpha(\gt q_\alpha)=\ad^*(\xi)(\alpha)(\gt q_\alpha)=-\alpha([\xi,\gt q_\alpha])=0.
\]
Suppose now that $\alpha(\gt q_\alpha)=0$. Then $\alpha\in\Ann(\gt q_\alpha)$. 
A standard fact is that $\Ann(\gt q_\alpha)=\ad^*(\gt q){\cdot}\alpha$, 
it follows from dimension reasons.  Hence $\alpha\in\ad^*(\gt q){\cdot}\alpha$. 
\end{proof}

We will use one standard characterisation of the contact Lie algebras. The statement is not difficult, but important, therefore an explanation is included. 
Suppose  $\dim\q=2n+1$, while $\ind\gt q=1$, and $\ap\in\q^*_{\sf reg}$. 
Then $\q^*$ has a basis  $\{\xi_1,\ldots,\xi_{2n+1}\}$ such that
$$
\textsl{d}_{\q}\ap=\xi_1\wedge\xi_2+\ldots+\xi_{2n-1}\wedge\xi_{2n}.
$$ 
Here 
$\gt q_\ap$ is equal to 
$\Ann(\langle \xi_1,\ldots,\xi_{2n}\rangle_{\bbk})\subset\q$. Note that  
$(\bwedge^n \textsl{d}_{\gt q}\alpha)\wedge\alpha = 0$ if and only if $\alpha \in\langle \xi_1,\ldots,\xi_{2n}\rangle_{\bbk}$. 
The point $\ap$ is contained in the subspace $\langle \xi_1,\ldots,\xi_{2n}\rangle_{\bbk}$ if and only if 
$\ap\in\Ann(\q_\ap)$, i.e., if 
$\ap(\q_\ap)=0$.  In other words, 
\begin{equation}\label{def0}
\text{$\ap\in\q^*_{\sf reg}$ is a contact form if and only if $\ap(\q_\ap)\ne 0$.}
\end{equation}   
Hence, for a Lie algebra $\q$ of index 1, we have  
\begin{equation}\label{def}
\gt q \text{ is contact} \ \Leftrightarrow \ \alpha(\q_\alpha)\ne 0 \text{ for a generic point } \alpha\in\gt q^*_{\sf reg}. 
\end{equation}

In most of the paper, we assume that $\bbk=\overline{\bbk}$. This is a natural assumption, since the property of being contact does not
change under field extensions. 

Whenever dealing with classical Lie algebras, we assume that $E_{ij}$ are elementary matrices (matrix units). 

If $V$ is a vector space and $k$ a natural number, then $kV$ is a direct sum of $k$ copies of $V$.

\section{Lie algebras of index 1}\label{sec1}

In this section, we obtain new characterisations of the contact Lie algebras. 

\begin{prop}\label{equiv1}
Suppose that $\ind\q=1$. Then $\q$
is contact if and only if $\ap\not\in\ad^*(\gt q){\cdot}\alpha$ for a generic point $\ap\in\gt q^*$. 
For an algebraic $\q$ over an algebraically closed field, this condition means that 
a generic $Q$-orbit in $\gt q^*$ is not conical. 
\end{prop}
\begin{proof}
By Lemma~\ref{alg-con}, $\ap$ is contained in $\ad^*(\gt q){\cdot}\alpha$ if and only if $\alpha(\q_\alpha)=0$. 
The desired equivalence follows now from~\eqref{def}. 

Note that in case $\bbk=\overline{\bbk}$, 
an orbit $Q\alpha\subset \gt q^*$ contains $\bbk^{\!^\times}\!\ap$ if and only if 
$\bbk\alpha\subset\ad^*(\gt q){\cdot}\alpha$. 
\end{proof}

Unless otherwise stated, we assume from now on that $\bbk=\overline{\bbk}$, $Q$ is a connected affine algebraic group, and $\gt q=\Lie Q$. 

\begin{thm}\label{equiv2}
A Lie algebra $\gt q$ of index 1 is contact if and only if 
there is a generic stabiliser for the coadjoint action. 
\end{thm}
\begin{proof}
Suppose first that $\gt q$ is contact. Choose $\alpha\in\gt q^*_{\sf reg}$ such that 
$\alpha(\gt q_\alpha)\ne 0$. Then the orbit $Q\alpha$ is not conical by Lemma~\ref{alg-con} and Proposition~\ref{equiv1}. 
Thereby 
\begin{equation}\label{Y}
Y=\bbk^{\!^\times}\! Q\alpha =\{ c \Ad\!(g)\alpha \mid c\in \bbk^{\!^\times}\!, g\in Q\}
\end{equation}
is a dense open subset of $\gt q^*$. Clearly $\q_\alpha$ and 
$\q_\gamma$ are conjugate  by an element of $Q$ for each $\gamma\in Y$. Thus, $\gt q_\alpha$ is a generic stabiliser for 
the coadjoint action. 

Suppose now that $\q$ is not contact, but $\gt q_\alpha$ with $\alpha\in\gt q^*_{\sf reg}$ is a generic stabiliser for 
the coadjoint action.  Let $y\in\gt q_\alpha$ be a non-zero vector. 
We show that $[\gt q,y]$ contains $y$. 

It suffices to prove that 
$\gamma(y)=0$ for every $\gamma\in\Ann([\gt q,y])$. 
Note that 
$$
\Ann([\gt q,y])=\{\beta\in\gt q^* \mid y\in\gt q_\beta\}.
$$ 
In particular, $\alpha\in \Ann([\gt q,y])$. Also
$\alpha$ is a regular point of $\gt q^*$, thereby $R=\gt q^*_{\sf reg}\cap  \Ann([\gt q,y])$ is a dense open subset 
of $\Ann([\gt q,y])$. For each point $\gamma\in R$, we have
$\gamma(\gt q_\gamma)=0$ by~\eqref{def0}. Since here $y\in\gt q_\gamma$ and $\dim\gt q_\gamma=1$, 
the equality $\gamma(y)=0$ holds. Clearly, it extends from $R$ to all points of $\Ann([\gt q,y])$. 

By a criterion \cite[Corollaire 1.8]{TYu} of Tauvel and Yu, $\gt q_\alpha\cap [\gt q,\gt q_\alpha]=0$, because   
$\gt q_\alpha=\bbk y$ is a generic stabiliser. This provides a contradiction, since 
$y\in\gt q_\alpha\cap [\gt q,\gt q_\alpha]$. 
\end{proof}

\subsection{Seaweeds and quasi-reductive Lie algebras}  \label{expl}
Let $\g=\Lie G$ be a simple Lie algebra and $\gt p_1,\gt p_2\subset\gt g$ two parabolic subalgebras such that 
$\gt p_1+\gt p_2=\g$. Then 
$\gt q=\gt p_1\cap\gt p_2$ is a Lie algebra of {\it seaweed type} or just a seaweed, also called a {\it bi-parabolic}.    
Each seaweed $\q$ is an algebraic Lie algebra and $\q=\Lie Q$, where $Q\subset G$ is the intersection of two 
parabolic subgroups. 
In \cite{Dima03-b}, Panyushev conjectured that if $\gt q_\gamma$ with $\gamma\in\q^*$
is a generic stabiliser for the coadjoint action of a seaweed $\q\subset\g$, then $Q_\gamma$ is reductive. 

If there is a stable point $\ap\in\q^*$, then $\q$ is called {\it stable} as well. This terminology is not standard, but it is used, for example, in \cite{coll}. 

Following \cite{duflo}, we say that $\q$ is  {\it quasi-reductive} if there is $\beta\in\q^*$ such that  the quotient 
$Q_\beta/Z$ by the centre $Z\subset Q$
 is a reductive subgroup of $\GL(\q^*)$. For more information on these Lie algebras see e.g. \cite{duflo,MRS}. 
 
\begin{rmk} \label{rem-qred} 
Suppose $Q_\beta/Z$ is a reductive subgroup of $\GL(\q^*)$. Then there is a generic stabiliser 
for the action of $Q/Z$ on $\q^*$, which is equal to a generic stabiliser for the action of $Q_\beta/Z$ on $\gt q_\beta^*$,
see e.g. \cite[Lemma~2.3]{MRS}.   Thus, a quasi-reductive  $\q$ is stable and its generic stabiliser 
 is a sum $ \Lie Z \oplus \Lie (\bbk^{\!^\times})^\ell $, where $\ell$ is the rank of $Q_\beta/Z$. 
\end{rmk}

Since the centre of a seaweed $\q\subset\g$ consists of semisimple elements, we may reformulate 
Panyushev's conjecture as follows: $\q$ is stable if and only if it is quasi-reductive.  The conjecture is proven by Ammari \cite{amm} 
via a case-by-case analysis. 
Combining this with Theorem~\ref{equiv2}, we derive the equivalence: a seaweed of index 1 is contact if and only if it is
quasi-reductive, which is the main result of \cite{coll}.  

The argument in \cite{coll} goes as follows:
\[
\text{stable} \ \overset{\text{Ammari}}{\Longrightarrow} \ \text{quasi-reductive}  \ \Longrightarrow \ \text{contact}  \ \Longrightarrow \  \text{stable}.
\]
The first paragraph of the proof of Theorem~\ref{equiv2} shows that the last implication is quite straightforward. 
The implications  
\begin{equation} \label{qred}
\text{quasi-reductive}  \  \overset{\text{\cite[Lemma~2.3]{MRS}}}{\Longrightarrow} \ \text{stable} \ 
\overset{\text{Theorem\,\ref{equiv2}}}{\Longrightarrow} \ \text{contact}
\end{equation}
 are true for all 
Lie algebras of index 1. 
It is possible to see directly, why a quasi-reductive $\q$ with $\ind\q=1$ has to be contact. 
An ingredient here is the fact that for a generic point $\ap\in\q^*$, its stabiliser $\q_\ap$ consists of $\ad$-nilpotent elements, if 
$\q$ is of index 1, but not contact. 

A contact Lie algebra does not have to be quiasi-reductive.
There are many  examples such that  $Q_\ap$ is unipotent and not central   for any contact form $\ap\in\q^*$. Below is one of them. 

\begin{ex} \label{dirpr}
Let $Q=\SL_2\ltimes\exp(\bbk^2)$ be a semi-direct product of $\SL_2$ and a two-dimensional Abelian unipotent group. 
Then $\ind\q=1$ and $Z=\{e\}$. Let $\{x,y\}$ be a basis of the nilpotent radical of $\q$ such that 
$[E_{12},x]=0$ and $[E_{12},y]=x$. Take $\ap\in\q^*$ such that $\ap(y)=1=\ap(E_{12})$ and  
$\ap(x)=\ap(E_{21})=\ap(E_{11}-E_{22})=0$. Then $\q_\ap$ is spanned by 
$u=E_{12}+2y$. Since $\ap(u)=3\ne 0$, the orbit $Q\ap$ is not conical and $\bbk u$ is a generic stabiliser for 
the coadjoint action. 
\end{ex}

We end this part of the paper by considering nilpotent Lie algebras.

\begin{ex}\label{nilp}
Let $\gt n=\Lie N$ be a nilpotent Lie algebra of index $1$.
The centre $\gt z$ of $\gt n$ is non-zero and $\gt z\subset\gt n_\ap$ for any $\ap\in\gt n^*$. 
Thus, $\gt n_\ap=\gt z$ for any $\ap\in\gt n^*_{\sf reg}$. Then $\gt z$ is the generic stabiliser for the coadjoint action of $\gt n$.
Clearly $\ap(\gt z)\ne 0$ for a generic $\ap$ and $\gt n$ is contact. 
Since $N_\ap$ is connected, it is equal to the centre $Z\subset N$. Hence $N_\ap/Z=\{e\}$ and $\gt n$ is quasi-reductive. 
 For instance, $\gt n$ may be a $3$-dimensional Heisenberg 
Lie algebra with a basis $\{x,y,z\}$ such that $z$ is central and $[x,y]=z$. 
\end{ex}

\subsection{Contact semi-direct products} \label{sec-products}
We consider a semi-direct product $Q=L\ltimes\exp(V)$, where  $\exp(V)$ is a normal Abelian unipotent subgroup.  
Let $\gamma\in V^*$ be a generic point and $L_\gamma\subset L$ its  stabiliser.   Set $\gt l_\gamma=\Lie L_\gamma$. By 
a formula of Ra{\"i}s \cite{r}, 
\begin{equation} \label{semi-ind}
\ind\gt q=\ind \gt l_\gamma+ \dim V-\dim L\gamma
\end{equation}
Suppose that $\ind\q=1$. Then there are two possibilities. Either 
\begin{itemize}
\item[{\tt (A)}:] $\dim L\gamma=\dim V-1$ and $\ind\gt l_\gamma=0$ \  or  
\item[{\tt (B)}:] $\dim L\gamma=\dim V$ and $\ind\gt l_\gamma=1$.
\end{itemize}
We want to check, whether $\q$ is contact or not. Regard $\gamma$ as a function on 
$\q$ such that $\gamma(\gt l)=0$ for $\gt l=\Lie L$. Also we identify $\gt l^*$  with $\Ann(V)\subset\q^*$. 

\begin{lm}\label{A}
In case {\tt (A)},   $\q$ is contact if and only if a generic $L$-orbit on $V^*$ is not 
conical. 
\end{lm}
\begin{proof}
By dimension reasons,   $\ad^*(V){\cdot}\gamma=\Ann(\gt l_\gamma)\subset\gt l^*$. 
Let $\ap\in\gt l^*$ be such that $L_\gamma \bar\ap$ is open in $\gt l_\gamma^*$ for 
the restriction $\bar\ap=\ap|_{\gt l_\gamma}$. This is an open condition on $\ap$. Set $\beta=\ap+\gamma\in\gt q^*$. 
Then $Q\beta$ is a dense open subset of 
$\gt l^*\times L\gamma$.  Hence $Q\beta$ is conical if and only if $L\gamma$ is conical. 
Now the result follows from Proposition~\ref{equiv1}. 
\end{proof}

In case {\tt (A)}, 
both instances, $L\gamma$ is conical or is not,  take places. 
The following example illustrates this. 
 
\begin{ex}\label{ex-k}
Consider $\gt q=\bbk\ltimes V$, where $[V,V]=0$, $\dim V=2$, and $\bbk$ acts on $V$ with two non-zero characters
$\lambda$ and $\mu$. Then $\gt q$ is contact if and only if 
$\lambda\ne \mu$.
\end{ex} 
 
The situation  {\tt (B)} is much more interesting. Suppose that  it takes place. 
Let $L_{<\gamma>}\subset L$ be the normaliser of the line $\bbk\gamma$ and set $\gt l_{<\gamma>}=\Lie L_{<\gamma>}$. 

\begin{thm} \label{impl} Let $Q=L\ltimes\exp(V)$ be such that $\ind\q=1$ and 
$L\gamma\subset V^*$ is open. \\[.2ex]
{\sf (i)} Suppose that $\gt l_\gamma$ is not contact. Then 
 $\q$ is contact if and only if  $\ind\gt l_{<\gamma>}=0$. \\[.2ex]
{\sf (ii)} If $\gt l_\gamma$ is contact, but $\q$ is not contact,  then $\ind\gt l_{<\gamma>}=0$.  
\end{thm}
\begin{proof}
Each generic $Q$ orbit in $\q^*$ contains a point $\ap$ such that $\ap|_{V}=\gamma$. 
Set $\beta=\ap|_{\gt l_\gamma}$. 
Without loss of generality assume that $\beta$ is a generic point of $\gt l_\gamma^*$. 
For case 
{\sf (ii)}, assume that $\beta$ is contact. 
Since $\overline{L\gamma}=V$, there is $x\in\gt l$ such that $\ad^*(x)(\gamma)=\gamma$. 
Clearly $\gt l_{<\gamma>}=\bbk x\oplus\gt l_\gamma$ and $[x,\gt l_\gamma]\subset\gt l_{\gamma}$. We say that $x$ acts on 
$\gt l_{\gamma}$ and on $\gt l_{\gamma}^*$.  

Since
$\ad^*(V){\cdot}\gamma=\Ann(\gt l_\gamma)\subset\gt l^*$, the orbit $Q\ap$ is conical if and only if 
$x{\cdot}\beta\in \beta +\ad^*(\gt l_\gamma){\cdot}\beta$. 

{\sf (i)} 
Since $\gt l_\gamma$ is not contact, $L_\gamma \beta\subset\gt l_\gamma^*$ is a conical orbit, 
see Proposition~\ref{equiv1}. Thus, $Q\ap$ is conical if and only if  $x{\cdot}\beta \in \ad^*(\gt l_\gamma){\cdot}\beta$.

If $x{\cdot}\beta \not \in \ad^*(\gt l_\gamma){\cdot}\beta$, then  $\dim L_{<\gamma>}\beta\ge 1+\dim L\beta$. Since each coadjoint orbit is even-dimensional, 
$\ind\gt l_{<\gamma>}\le (1+\dim\gt l)-(2+\dim L\beta)\le \ind\gt l-1=0$. 

Suppose that $x{\cdot}\beta \in \ad^*(\gt l_\gamma){\cdot}\beta$. 
Let $\tilde\beta\in \gt l_{<\gamma>}^*$ be an extension of $\beta$.  
Then $x{\cdot}\beta+\ad^*(\gt l_\gamma){\cdot}\beta=\ad^*(\gt l_\gamma){\cdot}\beta$ is the restriction 
of $\ad^*(\gt l_{<\gamma>}){\cdot}\tilde\beta$. Hence $\dim\ad^*(\gt l_{<\gamma>}){\cdot}\tilde\beta\le 1+\dim\ad^*(\gt l_\gamma){\cdot}\beta$. Because each coadjoint orbit is even-dimensional, 
we have $\dim\ad^*(\gt l_{<\gamma>}){\cdot}\tilde\beta=\dim\ad^*(\gt l_\gamma){\cdot}\beta$
and 
$\ind  \gt l_{<\gamma>}=(1+\dim\gt l) - \dim L_\gamma\beta=2$. 

{\sf (ii)} Since $\q$ is not contact, $Q\ap$ is conical by Proposition~\ref{equiv1}. Hence $x{\cdot}\beta\in \beta +\ad^*(\gt l_\gamma){\cdot}\beta$. 
If $x{\cdot}\beta \in \ad^*(\gt l_\gamma){\cdot}\beta$, then also $\beta \in\ad^*(\gt l_\gamma){\cdot}\beta$, which is a contradiction. 
Thereby again $\dim L_{<\gamma>}\beta\ge 1+\dim L\beta$. 
\end{proof}

Part~{\sf (ii)} of Theorem~\ref{impl} states that if $\gt l_\gamma$ is contact and $\ind \gt l_{<\gamma>}\ne 0$, then 
$\gt q$ is contact.  If $\gt l_\gamma$ is contact and $\ind \gt l_{<\gamma>}=0$, then $\gt q$ may be contact as well. 
This can be checked with the help of a special semisimple element $s\in    \gt l_{<\gamma>}$. 

Let $\gt h=\Lie H$ be a Lie algebra of index zero and $H\beta\subset\gt h^*$ the open orbit. For 
each $\eta\in H\beta$ there is a unique element $s_\eta\in\gt h$ such that $\ad^*(s_\eta)\eta=\eta$. Clearly each  
$s_\eta$ is $H$-conjugate to $s:=s_\beta$.  The semisimpe part of $s$ multiplies $\beta$ by $1$. 
Since $s$ is unique, it 
 is $\ad$-semisimple. This $s$, introduced in \cite{Ooms}, is called a {\it principal element} of $\gt h$.  

\begin{prop}\label{impl-ii}
Suppose that $\gt l_\gamma$ is contact and $\ind \gt l_{<\gamma>}=0$. Let $s\in \gt l_{<\gamma>}$ be a 
principal element. Then $\gt q$ is contact if and only if $\ad^*(s)\gamma\ne\gamma$.
\end{prop}
\begin{proof}
We keep notation introduced in the proof of Theorem~\ref{impl}. Recall that 
$\gt q$ is not contact if and only if $x{\cdot}\beta \in \beta+\ad^*(\gt l_\gamma){\cdot}\beta$. 
Let again $\tilde\beta\in\gt l_{<\gamma>}^*$ be an extension of $\beta$ and let
$s\in \gt l_{<\gamma>}$ be the principal element associated with $\tilde\beta$. Since 
$\gt l_\gamma$ is an ideal of $\gt l_{<\gamma>}$, the element $s$ acts on $\gt l_\gamma^*$ and 
$s{\cdot}\beta=\beta$. Since the orbit $L_\gamma\beta$ is not conical, we have $s\not\in\gt l_\gamma$, i.e., 
$s\in ax+\gt l_\gamma$ for some $a\in\bbk^{^\times}$.  
Now 
\[
x{\cdot}\beta \in \beta+\ad^*(\gt l_\gamma){\cdot}\beta \ \Leftrightarrow \ a=1 \ \Leftrightarrow \ \ad^*(s)\gamma=\gamma.
\]
This finishes the proof. 
\end{proof}

\begin{ex} \label{ex-impl}
We apply Proposition~\ref{impl-ii} to the Lie algebra $\gt q=\bbk\ltimes V$ of Example~\ref{ex-k}.
Let $x,y\in V$ be linearly independent eigenvectors of $\bbk$. We can decompose $\gt q$ as
$\gt l\ltimes V_1$, where $\gt l=\bbk\ltimes \bbk x$ and $V_1=\bbk y$. 
Let $\gamma\in V_1^*$ be a non-zero point. 
Then 
$\gt l_\gamma=\bbk x$ is contact and 
$\gt l=\gt l_{<\gamma>}$ is of index zero.  For a principal element $s\in\gt l$, we have 
$[s,x]=-x$. Thus, $\ad^*(s)\gamma=\gamma$ if and only if $[s,y]=-y$, i.e., if 
the characters $\lambda$ and $\mu$ are equal.  
\end{ex}

Further examples of semi-direct products will appear in Section~\ref{sec-aff}. 

Let $\q$ be a $3$-dimensional Lie algebra with a basis $\{s,x,y\}$ such that $[s,x]=x$, $[s,y]=y$, and $[x,y]=0$.
It appeared in Example~\ref{ex-impl}. The decomposition $\gt q=\gt l\ltimes V$ with 
$\gt l=\langle s,x\rangle_{\bbk}$, $V=\bbk y$ satisfies the assumption that $L$ acts on $V^*$ with an open orbit $L\gamma$.
As we have seen, $\gt l_\gamma=\bbk x$ is contact and there is a generic stabilisier for its coadjoint action, while 
$\gt q$ is not contact and its  coadjoint action has no generic stabiliser. 
Furthermore, $\gt l_\gamma$ is quasi-reductive, while $\gt q$ is not.

If we impose a stronger condition on the  stabiliser $L_\gamma$, then the reduction from $Q$ to $L_\gamma$ works better.
Following \cite{MRS}, we say that $\gt h=\Lie H$ is {\it strongly quasi-reductive}, if $\gt h$ is quasi-reductive
and the centre $Z\subset H$ consists of semisimple elements, in other words, if $H_\beta$ is reductive for
some $\beta\in\gt h^*$. In certain cases, this definition depends on the group $H$. Although 
$\Lie\!(\bbk,+)\cong\Lie\!(\bbk^{\!^\times},\times)$, the first one is not strongly quasi-reductive, while the second is. 
Recall from Section~\ref{expl}, see Remark~\ref{rem-qred}, that if $H_\beta$ is reductive for one $\beta$, then $H_\ap$ is reductive for 
a generic point $\ap\in\gt h^*$. 

In the following proposition, we do not assume that $\ind\gt q=1$.

\begin{prop}[{cf. \cite[Corollary 2.9.]{Dima03-b}, \cite[Lemma\,4.3]{MRS}}] \label{impl2}
Let $Q=L\ltimes\exp(V)$ be such that $L\gamma\subset V^*$ is an open orbit.  
Then $\q$ is strongly quasi-reductive if and only if 
$\gt l_\gamma=\Lie L_\gamma$ is strongly quasi-reductive. 
\end{prop}
\begin{proof}
Each generic $Q$ orbit in $\q^*$ has a point $\ap$ such that $\ap|_{V}=\gamma$. Clearly $Q_\ap\subset L_\gamma\ltimes\exp(V)$.
We have 
\[
\exp(V)\ap=\ap+\ad^*(V){\cdot}\gamma,    \ \ \text{ where }  \  \ad^*(V){\cdot}\gamma|_V=0.
\] 
Furthermore, 
 $\ad^*(V){\cdot}\gamma=\Ann(\gt l_\gamma)\subset\gt l^*$ is of dimension $\dim V$
 and $\ad^*(v)(\gamma)\ne 0$ for each non-zero $v\in V$. 
Hence $Q_\ap  =Q_\ap/(Q_\ap\cap\exp(V)) \cong (L_\gamma)_\beta$ for the restriction $\beta=\ap|_{\gt l}$. Since $\beta$ is a generic point of $\gt l^*$, we are done.  
\end{proof}

Recall that a (strongly) quasi-reductive Lie algebra of index $1$ is contact, see~\eqref{qred}. 
Therefore Theorem~\ref{impl} and Propositions~\ref{impl-ii}, \ref{impl2} can be employed for obtaining a classification of contact semi-direct products
$L\ltimes\exp(V)$.  Under some restrictions on $L$ and $V$, a reasonable description is probably possible. 
Similar results have been obtained before. 
The semi-direct products $L\ltimes\exp(V)$ of index zero, where $L$ is reductive,
are classified in \cite{El-fr} under the assumption that 
$V$ is a simple $L$-module or the commutator group of $L$ is simple. 

\section{Invariants and semi-invariants} \label{sec-inv}

The symmetric algebra $\gS(\q)$ is identified with the graded ring $\bbk[\q^*]$ of polynomial functions on $\q^*$. The adjoint action of $\q$ extends to a representation of 
$\q$ on $\gS(\q)$. Since $Q$ is connected, we have 
$\bbk[\q^*]^Q=\gS(\q)^{Q}=\gS(\q)^{\q}$. The elements of this ring are called {\it symmetric invariants} of $\gt q$. 
For $F\in\gS(\q)$,
we have 
\[
F\in\gS(\q)^{\q} \ \Leftrightarrow \ \xi{\cdot}F=0 \ \forall \xi\in\q.
\] 
Then $F\in\gS(\q)$ is called a {\it semi-invariant of $\q$} if $QF\subset\bbk F$ or equivalently if 
$\xi{\cdot}F\in\bbk F$ for each $\xi\in\q$. Let $\gS(\gt q)_{\sf si}\subset\gS(\q)$ be the subring generated by the 
semi-invariants of $\q$.  Note that both $\gS(\q)^{\q}$ and $\gS(\gt q)_{\sf si}$ are graded subalgebras of $\gS(\q)$.  

\subsection{General facts}  \label{gen}
Let $f\in\bbk(\q^*)^Q$ be a rational $Q$-invariant. Then $f=u/v$ for some coprime polynomials $u,v\in\bbk[\q^*]$. 
Since $\bbk[\gt q^*]$ is a unique factorisation domain, such a presentation is unique up to scalar multiples. Therefore $u$ and $v$ are semi-inva\-riants of $Q$.  

Suppose $\ind\gt q=1$. Then $\trdeg \bbk(\q^*)^Q=1$ by~\eqref{tr-ind}. Since   
$\Quot\!(\bbk[\q^*]^Q)$ is contained in $\bbk(\q^*)^Q$, we obtain $\trdeg\gS(\q)^{\q}\le 1$. It may happen that 
$\gS(\q)^{\q}=\bbk$. Suppose $\gS(\q)^{\q}\ne\bbk$. 
Then $\gS(\q)^{\q}$ is generated by one element. Next we give an explanation of this standard fact. 

Let $F,F_1\in \gS(\q)^{\q}$ be two non-constant homogeneous polynomials. 
Then they are algebraically dependent and hence $F^a=sF_1^b$ for some coprime natural numbers $a,b$ and $s\in\bbk^{\!^\times}$. 
Thereby $H=\sqrt[b]{F}\in\bbk[\q^*]$ and the orbit $QH$ has at most $b$ elements. Since the group $Q$ is connected, $H\in\bbk[\q^*]^Q$. 
If  $F$ is an invariant of the minimal degree, then $F_1\in\bbk[F]$. Thus $\bbk[\q^*]=\bbk[F]$. 


In any case, $\bbk(\gt q^*)^Q\ne \bbk$. There is at least 
one  rational invariant $f=u/v\not\in\bbk$, where $u,v\in\gS(\q)_{\sf si}$.   Hence $\gS(\q)_{\sf si}\ne\bbk$. 
In view of classical results of L\"uroth and Castelnouvo, $\bbk(\gt q^*)^Q=\bbk(f)$ for some $f$. 
In Section~\ref{sec-set}, we explain how to find such a generator for a contact $\q$.

For any $\q$, 
let $\cW=\bigoplus_{i= 0}^{\dim\q} \cW^i$ be the graded skew-symmetric algebra of polynomial polyvector fields on $\q^*$. Over 
 $\gS(\q)$ it is generated by 
$\{\partial_i=\partial_{x_i}| 1\le i \le N\}$, where $N=\dim\q$ and%
$\{x_i\mid 1\le i\le N\}$ is a basis of $\q$. 
Let 
$$\pi=\sum_{i<j} [x_i,x_j] \partial_i \wedge \partial_j\in \cW^2$$ 
be the {\it Poisson tensor} of $\q$. 
Evaluating  $\pi$ at a point $\gamma\in\q^*$, we obtain $\pi(\gamma)=\textsl{d}_{\q}\gamma$. 
Then $d=\frac{1}{2}(\dim\q-\ind\q)$ is the largest number such that $\bwedge^d \pi\ne 0$. 

\begin{df}[{cf.~\cite[Sect.\,1]{fonya-et}\,\&\,\cite[Sect.\,4.1]{js}}] \label{def-fund}
A {\it fundamental semi-invariant} of $\gt q$ is a polynomial $\bp$ such that 
$\bwedge^{d}\pi=\bp {\eus R}$ with ${\eus R}\in \cW^{2d}$ and the zero set of ${\eus R}$ in $\gt q^*$  
has codimension grater than or equal to $2$. 
\end{df}

The zero set of $\bp$ is the union of divisors contained in $\q^*_{\sf sing}$. Thereby $\bp$ is indeed 
a semi-invariant of $\q$. 
By the construction, $\q$ has the codim--$2$ property if and only if $\bp=1$.

\begin{rmk} \label{semi}
Let us show that $\gS(\q)_{\sf si}\ne\bbk$ for 
 any $\q\ne 0$. Indeed,  if $\ind\q\ge 1$, then the argument with the inclusion 
$\bbk(\q^*)^Q\subset \Quot\gS(\q)_{\sf si}$ works. 
If $\ind\q=0$, then $N=2d$ and $\deg \bp = d \ge 1$ for a non-zero $\q$. 
The Lie algebras of index zero are called 
{\it Frobenius}. 
\end{rmk}

\subsection{Importance of regular invariants} 

\begin{prop}\label{inv1}
 If\/ $\ind\q=1$ and $\gS(\gt q)^{\gt q}=\bbk[F]$\/ for a non-constant $F$, then $\gt q$ is contact. 
\end{prop}
\begin{proof}
The invariant $F$ is constant on each orbit $Q\alpha\subset\gt q^*$. Hence $Q\alpha$ cannot be conical, if 
$F(\alpha)\ne 0$. The rest follows from Proposition~\ref{equiv1}. 
\end{proof}

\begin{cl} \label{no}
If\/  $\ind\gt q=1$ and the  
radical of $\gt q$ consists of $\ad$-nilpotent elements, then $\gt q$ is contact.
\end{cl}
\begin{proof}
Let $\mathrm{Rad}(\gt q)\lhd \q$ be the  radical of $\q$. 
Each  $x\in\mathrm{Rad}(\gt q)$ is $\ad$-nilpotent, hence it acts as a nilpotent element on every   $\gS^k(\gt q)$. 
Therefore each semi-invariant of $\gt q$ in $\gS(\gt q)$ is invariant under $x$. 
In other words, $\mathrm{Rad}(\gt q)$ acts trivially on each $\q$-stable line $\bbk u\subset\gS(\q)$. 
Then $\bbk u$ gives rise to a $1$-dimensional representation of the semisimple algebra 
$\q/\mathrm{Rad}(\gt q)$. 
Thus, each semi-invariant of $\gt q$ is an invariant 
and $\gS(\gt q)^{\gt q}\ne\bbk$. 
\end{proof}

Extending the ground field we do not violate the conditions of Corollary~\ref{no}. Hence its statement holds over any 
field.  

\begin{ex} \label{ex-ohne}
There are contact Lie algebras with no regular invariants.  \\[.2ex]
{\sf (a)} A Borel subalgebra  $\gt b\subset\gt{sl}_3$ 
is contact. A generic stabiliser for $\gt b^*$ is a subspace of the Cartan subalgebra that commutes with the centre of the nilpotent 
radical. One sees easily that $\gS(\gt b)^{\gt b}=\bbk$. 
\\[.3ex]
{\sf (b)} Consider a maximal parabolic subalgebra $\gt p\subset\gt{sp}_4$ such that 
$\gt p=\gt{gl}_2\ltimes\gt n$, where the nilpotent radical $\gt n$ is a $3$-dimensional Heisenberg Lie algebra. 
It is contact and a generic stabiliser
for $\gt p^*$ is a  Cartan subalgebra of $\gt{sl}_2$. 
Proper parabolic subalgebras of simple Lie algebras have only trivial symmetric invariants,
see e.g. \cite[Corollary 2.11]{Dima03-b}. 
\end{ex}

\subsection{The ring generated by semi-invariants} \label{gen-semi}

\begin{prop} \label{free}
Let $\gt q$ be a contact Lie algebra. Then 
$\gS(\gt q)_{\sf si}$ has an algebraically independent set of generators
$\{H_1,\ldots,H_m\}$ such that each $H_i$ is
an irreducible homogeneous semi-invariant.
Furthermore, the fundamental semi-invariant $\bp$ of $\gt q$ is equal to 
$\prod_{i=1}^m H_i^{a_i}$ with $a_i\ge 0$.
\end{prop}
\begin{proof}
Set $\widetilde{Q}=Q\times\bbk^{\!^\times}$. By Proposition~\ref{equiv1}, this group acts on $\gt q^*$ with an open orbit. 
Then by~\cite{SK}, the algebra ${\mathcal A}\subset\bbk[\gt q^*]$
generated by  $\widetilde{Q}$-semi-invariants  is a polynomial ring. More precisely,
let $\widetilde{Q}\alpha\subset \gt q^*$ be the open $\widetilde{Q}$-orbit.
It appeared in  the proof of Theorem~\ref{equiv2} as the subset $Y$, see~\eqref{Y}. 
Let $D_1,\dots, D_m$ be all simple divisors in $\gt q^*\setminus\widetilde{Q}\alpha$\, (the irreducible components of codimension $\ge 2$ in $\gt q^*$ are not needed). 
If $D_i=\{H_i=0\}$ and $H_i\in\bbk[\gt q^*]$ is irreducible, then
$H_i$ is a semi-invariant of $\widetilde{Q}$ and $H_1,\dots,H_m$ freely generate ${\mathcal A}$, see \cite[\S\,4]{SK}. 

Clearly each semi-invariant of $\widetilde{Q}$ is also a semi-invariant of $Q$, hence ${\mathcal A}\subset \gS(\gt q)_{\sf si}$. 
 Suppose that $Q F\subset \bbk^{\!^\times}\! F$ for some $F\in \gS(\gt q)$. 
Then each homogeneous component of $F$ is also a semi-invariant of $Q$, and hence of $\widetilde{Q}$. 
This leads to the equality  $\gS(\gt q)_{\sf si}={\mathcal A}=\bbk[H_1,\dots,H_m]$.

Note that $\widetilde{Q}\alpha\subset \gt q^*_{\sf reg}$. Let $p_i$ be a prime factor of $\bp$. Then 
the zero set of $p_i$ is a simple divisor in the complement of $\widetilde{Q}\ap$. Hence $p_i$ is equal to some $H_j$ up to a non-zero 
scalar. 
\end{proof}

\begin{rmk}
If $\alpha\in\gt q^*$ is a contact linear  form, then $\widetilde{Q}\alpha$ is a dense open subset of $\gt q^*$. This implies that 
a contact form is unique up to conjugation by elements of $Q$ and scalar multiplication. Over a non-closed field,  the situation is different. For instance, take $Q=\SL_2(\R)\ltimes\exp(2\R^2)$. Then  
$\q$ is contact. The group $\widetilde{Q}=Q\times\R^{\!^\times}$ has two different orbits  on $\q^*$ that consists of contact
points. Studying contact orbits over $\R$ 
 is an excellent direction for further research. 
\end{rmk}

Unlike a Frobenius Lie algebra, a contact one may have the codim--2 property.  
For instance, $\gt g=\gt{sl}_2$ is contact. The only singular point in $\gt g^*\cong\gt g$ is zero. 
A non-zero semisimple element defines a contact form, while the nilpotent orbit is conical, but  has dimension $2$ as well.  

\begin{rmk} \label{f}
A contact Lie algebra $\gt q$ has another very important semi-invariant, which is never a constant. It is described in matrix terms in \cite[Sect.\,4]{sal}. 
Set $v=\sum_{i=1}^{\dim\q} x_i \partial_i\in\cW^1$. 
Then $v(\gamma)=\gamma$ for any $\gamma\in\q^*$.  Recall that $\pi(\gamma)=\textsl{d}_{\q}\gamma$. 
Set $n=\frac{1}{2}(\dim\q-1)$. 
The exterior product  $(\bwedge^n \pi) \wedge v$ is non-zero at $\ap\in\q^*$ if and only if $\ap$ is contact.  
We have also   
$$
(\bwedge^n \pi) \wedge v = \fb \partial_1\wedge\ldots\wedge \partial_{2n+1},
$$
where $\fb$ is a homogeneous polynomial of degree $n+1$.  
The zero set of $\fb$ is exactly the complement $\gt q^*\setminus\widetilde{Q}\alpha$, where $\ap$ is a contact form. 
Then $\fb=c\prod_{i=1}^m H_i^{a_i}\in\gS(\q)_{\sf si}$ with $a_i\ge 1$ and $c\in\bbk^{\!^\times}\!$. 
Since $\bwedge^{n}\pi=\bp {\eus R}$ with ${\eus R}\in \cW^{2n}$, we have 
$(\bwedge^n \pi) \wedge v=\bp  {\eus R} \wedge v$ and 
$\bp|\fb$. 
\end{rmk}

\begin{ex} \label{ex-f}
(1)\, If $\gt g=\gt{sl}_2$, then $\fb$ is the determinant (up to a non-zero scalar). 

\noindent
(2)\, For a Heisenberg Lie algebra $\gt q$ of dimension $2n{+}1$ with the centre $\bbk z$, 
we have
$\bp=z^n$ and $\fb=z^{n+1}$.  
\end{ex}

\begin{ex} \label{2}
Consider Lie algebras of Example~\ref{ex-ohne}. 
For $\gt b=\Lie B$, we choose a basis consisting of 
\[
h=\diag(1,-2,1), \ h_1=\diag(1,0,-1), \ x=E_{12}, \ y=E_{23}, \ z=[x,y]=E_{13}.
\]
Then $\gS(\gt b)_{\sf si}=\bbk[z,H_2]$ with $H_2=hz+3xy$. If $\gamma(z)\ne 0$ for $\gamma\in\gt b^*$, then 
$B\gamma$ contains a point $\tilde\gamma$ such that $\tilde\gamma(z)=\gamma(z)$ and 
$\tilde\gamma(x)=\tilde\gamma(y)=0$. Clearly $\gt b_{\tilde\gamma}=\bbk h$ and $\gamma\in\gt b^*_{\sf reg}$. 
Therefore $\gt b^*_{\sf sing}$ is contained in $D_1=\{z=0\}$.  However, 
$\gt b_\beta=\bbk z$ for any $\beta\in D_1$ such that $\beta(x)\ne 0$ and $\beta(y)\ne 0$. 
Thus, $\bp=1$. Since $\deg\fb=3$, we have $\fb \in \bbk z H_2$, actually $\fb=2 z H_2$. 

For $\gt p\subset\gt{sp}_4$ from part {\sf (b)}, we have  
$\gS(\gt p)_{\sf si}=\bbk[z,H_2]$, where $z$ is again a non-zero element in the centre of the nilpotent 
radical and $\deg H_2=3$. Here $\fb=z H_2$ up to a scalar  and   $\deg \bp\le 3$, since $\dim\gt b=7$. 
The semi-direct product $\gt{gl}_2\ltimes\bbk^2$ is a Frobenius Lie algebra. This implies 
$\gt p_\beta=\bbk z$ for a generic point of 
$D_1=\{z=0\}$. In a generic $Q\gamma\subset D_2$, there is a point $\tilde\gamma$ such that 
$\tilde\gamma(z)=1$, $\tilde\gamma(\bbk^2)=0$, and $\bar\gamma=\tilde\gamma|_{\gt{sl}_2}$ is a non-zero nilpotent element.   
Then $\gt p_{\tilde\gamma}=(\gt{sl}_2)_{\bar\gamma}$ is of dimension $1$. Therefore 
$D_1$ and $D_2$ are not contained in $\gt p^*_{\sf sing}$ and $\bp=1$.
\end{ex}

\begin{ex}\label{not-free}
There are Lie algebras of index 1 such that $\gS(\gt q)_{\sf si}$ is not a polynomial ring.  
Consider first a semi-direct product $\gt s=\gt{sl}_2\ltimes 4 \bbk^2$, where $4\bbk^2$ is an Abelian ideal. 
Then $\dim\gt s=11$ and  generic 
$\SL_2$-orbits on $4\bbk^2$ are of dimension $3$. 
By~\eqref{semi-ind}, 
$\ind\gt s=5$. The ring $\gS(\gt s)^{\gt s}=\bbk[(4\bbk^2)^*]^{\SL_2}$ is generated by $\binom{4}{2}=6$ invariants of degree 2.  
There is one relation among them. We extend $\gt s$ to $\widetilde{\gt s}$ by adding 
a Cartan subalgebra $\gt t\subset\gt{gl}_4$, where $\gt{gl}_4$ is the centraliser  of $\gt{sl}_2$ in $\gt{gl}(4\bbk^2)$. 
A generic stabiliser for the obtained action of $(\bbk^{\!^\times})^4\times \SL_2$ on $4\bbk^2$ is still trivial.
Hence $\ind\widetilde{\gt s}=8-7=1$.  Here $\gS(\widetilde{\gt s})_{\sf si}=\gS(\gt s)^{\gt s}$ is not a polynomial ring. 
There is a similar example for each $n\ge 1$, namely, 
any semi-direct product $\gt q=(\gt{sl}_{2n+2}\oplus 4\bbk)\ltimes 4\bbk^{2n+2}$.
Here $\ind\gt q=1$ and $\gS(\gt q)_{\sf si}=\gS(\gt s_n)^{\gt s_n}$ for 
$\gt s_n=\gt{sl}_{2n+2}\ltimes 4\bbk^{2n+2}$. 
By 
 \cite[Thm.\,4.3]{pisa}, $\gS(\gt s_n)^{\gt s_n}$ is not a polynomial ring. 
\end{ex} 

There is no equivalence in Proposition~\ref{free}, i.e., $\gS(\q)_{\sf si}$ may be a polynomial ring for a non-contact Lie algebra of index $1$. Let $\gt q=\bbk\ltimes V$ be the Lie algebra from Example~\ref{ex-k}, where 
$\bbk$ acts on $V$ with two non-zero characters
$\lambda$ and $\mu$. Then $\gS(\gt q)_{\sf si}=\gS(V)$ is always a polynomial ring. However, $\gt q$ is contact if and only if 
$\lambda\ne \mu$.

\subsection{The set of semi-invariants}  \label{sec-set}
Suppose  that $\q$ is contact. In this section, we give a criterion for the existence of non-trivial symmetric invariants of 
$\gt q$, describe a generator of $\bbk(\gt q^*)^{\gt q}$, and say a few words about the set, not the ring, of $\gt q$-semi-invariants. 

 Let $H_1,\ldots,H_m$ be the generators of the ring $\gS(\q)_{\sf si}$ described in 
the proof of Proposition~\ref{free}. 
Let $\chi_i: \widetilde{Q}\to \bbk^{\!^\times}$ be the $\widetilde{Q}$-character corresponding to $H_i$. 
We have $g{\cdot}H_i=\chi_i(g)H_i$ for all $g\in \widetilde{Q}$. Then the differentials  $\textsl{d}\chi_i\!:(\q\oplus\bbk)\to \bbk$
with $1\le i\le m$ are linearly independent,
because $\bbk(\q^*)^{\widetilde{Q}}=\bbk$, 
see also  \cite[\S\,4]{SK}. 

Let $\bar\chi_i$ be the restriction of $\textsl{d}\chi_i$ to $\q$. Then 
$\dim\langle\bar\chi_1,\ldots,\bar\chi_m\rangle_{\bbk}=m-1$. 
Since the character group ${\eus X}(\widetilde{Q})$ is a lattice, any relation between 
$\bar\chi_i$ has integer coefficients up to scalar multiples.  Up to a suitable enumeration, the minimal relation looks as 
\begin{equation} \label{rel}
\sum_{i=1}^a c_i\bar\chi_i = \sum_{i=a+1}^b c_i \bar\chi_i,
\end{equation}
where $c_i\ge 1$ for each $i$ and $\gcd(c_1,\ldots,c_b)=1$. Assume that $a\ge (b-a)$. If $b=a$, then 
$F=\prod_{i=1}^a H_i^{c_i}\in\bbk[\q^*]^Q$ and $\bbk(\q^*)^Q=\bbk(F)=\Quot\!(\bbk[\q^*]^Q)$. 
If $b>a$, then $\bbk[\q^*]^Q=\bbk$ and 
$f=(\prod_{i=1}^a H_i^{c_i})( \prod_{i=a+1}^b H_i^{-c_i})$ is a non-regular generator of $\bbk(\q^*)^Q$.  

\begin{rmk} \label{descr-si}
Let $H\in\gS(\q)_{\sf si}$ be homogeneous. Then the zero set $D$ of $H$ is a $\widetilde{Q}$-stable proper closed subset of $\q^*$.
Hence $D\cap\widetilde{Q}\alpha=\varnothing$ for the open orbit $\widetilde{Q}\alpha\subset\q^*$ and $H=\prod_{i=1}^m H_i^{a_i}$ with $a_i\ge 0$. 
 There are also non-homogeneous semi-invariants. For instance, $z+H_2$ 
is a  semi-invariant of
 the Lie algebra from Example~\ref{ex-ohne}{\sf (a)}, see Example~\ref{2}. 
\end{rmk}

\subsection{The canonical truncation} \label{sec-tr}
Let $\q_{\rm tr}\subset\q$ be the intersection of all kernels $\ker\bar\chi_i$ of the characters $\bar\chi_i$ defined in Section~\ref{sec-set}.
Then $\dim\q_{\rm tr}=\dim\q-m+1$. This subalgebra is called {\it the canonical truncation of $\q$}. 
Its crucial feature is that $\gS(\q)_{\sf si}\subset\gS(\q_{\rm tr})$, see~\cite[Kap.\,II,\,\S\,6]{bgr}. 
Actually, $\gS(\q)_{\sf si}\subset\gS(\q_{\rm tr})^{\q_{\rm tr}}$. 
By the construction, $\q_{\rm tr}\lhd\q$ and the quotient $\gt q/\q_{\rm tr}$ is Abelian. 
Furthermore, $\gt q=\q_{\rm tr}\oplus \widetilde{\gt t}$, where  $\widetilde{\gt t}$ is an Abelian subalgebra 
consisting of $\ad$-semisimple elements, but not necessarily  an ideal.  
Therefore $\q$ acts diagonalisably on $\gS(\q_{\rm tr})_{\sf si}$. Hence 
\[
\gS(\q_{\rm tr})_{\sf si} \subset \gS(\q)_{\sf si}  \subset\gS(\q_{\rm tr})^{\q_{\rm tr}}.
\]
Thus, $\gS(\q_{\rm tr})_{\sf si}=\gS(\q_{\rm tr})^{\q_{\rm tr}}= \gS(\q)_{\sf si} =\bbk[H_1,\ldots,H_m]$, 
see Proposition~\ref{free}.
These equalities imply that 
$\bbk(\q_{\rm tr}^*)^{\q_{\rm tr}}=\Quot\!(\bbk[\q_{\rm tr}^*]^{\q_{\rm tr}})$ and 
$m=\trdeg \bbk[\q_{\rm tr}^*]^{\q_{\rm tr}}=\ind\q_{\rm tr}$. 
 A general formula obtained in \cite[Lemma\,3.7]{fonya-et} brings the same result, 
 $\ind \q_{\rm tr}=\ind\gt q+(m-1)=m$. 
 
We have just seen that if $\q$ is contact, then $\q_{\rm tr}$ has the following  remarkable property: 
the
 ring $\gS(\q_{\rm tr})^{\q_{\rm tr}}$  is freely generated by 
$m=\ind\gt q_{\rm tr}$ elements.  
Repeating the argument of \cite[Remark\,3.1]{kot-T} we now obtain further results of this sort.   
Note that $\q_{\rm tr}$ has no proper semi-invariants in $\gS(\q_{\rm tr})$. Then by \cite[Prop.\,5.2]{js}
the differentials $\textsl{d}H_i$ are linearly independent in codimension $2$, i.e., 
\[
\dim \{\xi\in\q_{\rm tr}^*\mid \textsl{d}_\xi H_1\wedge\ldots\wedge\textsl{d}_\xi H_m=0\}\le \dim\q_{\rm tr}-2.
\]
By \cite[Thm.\,2.2]{kot-T}, each finite-dimensional  quotient 
 $\gt w_k:=\gt q_{\rm tr}[t]/(t^k)$  of the current algebra $\gt q_{\rm tr}[t]$ has the same property:  
the ring of symmetric invariants of $\gt w_k$ is freely generated by 
$k{\cdot}\ind\q_{\rm tr}=\ind\gt w_k$ elements.

\section{Affine seaweed subalgebras} \label{sec-aff}

Let $\g$ be a simple finite-dimensional non-Abelian Lie algebra with $r=\rk\gt g$.
We fix a Borel subalgebra $\gt b\subset \gt g$ and a Cartan subalgebra $\gt t\subset\gt b$. 
Let $\Pi=\{\alpha_1,\ldots,\alpha_r\}$ be the set of simple roots associated with 
$(\gt b,\gt t)$. For any root $\alpha$ of $(\gt g,\gt t)$, let $e_\alpha\in\gt g$ be a root vector. 
To a subset $S\subset \Pi$ one associates a standard parabolic subalgebra
$\gt p(S)$, which is generated by $\gt b$ and $\{e_{-\ap}\mid \ap\in S\}$.  Then let 
$\gt p(S)^-$ be the opposite parabolic, which is generated by $\gt t$ together with  $\{e_\ap\mid \ap\in S\}$ 
and $\{e_{-\ap} \mid \ap\in\Pi\}$. A seaweed in $\gt g$ is called {\it standard} if it is of the form 
$\gt p(S)\cap \gt p(T)^-$ for two subsets $S,T\subset \Pi$.  By \cite{Dima}, 
any seaweed in $\gt g$ is conjugate to a standard one. 

We will consider standard seaweeds of the loop 
algebra $\gt g[t,t^{-1}]=\gt g\otimes\bbk[t,t^{-1}]$.  
Let $\tilde \Pi=\{\ap_0\}\sqcup \Pi$ be  the 
affine root system associated with $\Pi$ 
and 
$\delta$ the highest root of $\g$.     
Set $\widehat{\gt b}=\gt b\oplus t\g[t]$ and $e_{-\alpha_0}=e_{\delta} t^{-1}\in\g[t^{-1}]$.  
To a subset $S\subset\tilde\Pi$, one associates a standard parabolic $\gt p=\gt p(S)\subset\gt g[t,t^{-1}]$,
which is generated by $\wb$ and $\{e_{-\beta}\mid \beta\in S\}$.  

Let $\gt p_{\Pi}(S)$  be the standard parabolic  of $\gt g$ associated with $S\cap\Pi$. 
For any Lie algebra $\gt q$, let $\U(\gt q)$ be its enveloping algebra. 
We have 
\[
\gt p=t\g[t]\oplus \gt p_{\Pi}(S) \oplus \gt p_{-1} \oplus [\gt p_{-1},\gt p_{-1}]\oplus [ [\gt p_{-1},\gt p_{-1}],\gt p_{-1}] \oplus \ldots ,
\]
where 
 $\gt p_{-1}=0$ if $\alpha_0\not\in S$ and $\gt p_{-1}=\U(\gt p_{\Pi}(S))e_{-\alpha_0}$ is a cyclic 
$\U(\gt p_{\Pi}(S))$-module if $\ap_0\in S$. 
Let  $\omega$ be an involution of $\gt g[t,t^{-1}]$ such that $\omega|_{\gt t}=-\id_{\gt t}$ and
$\omega(e_{\alpha} t^k)=e_{-\alpha} t^{-k}$ for all simple roots $\alpha\in\Pi$ and all $k\in\mathbb Z$. 
Then the opposite parabolic $\gt p^-$ is defined by $\gt p^-=\omega(\gt p)$. 

If $S=\tilde\Pi$, then $\gt p(S)=\g[t,t^{-1}]$. 
If $\ap_0\not\in S$, then $\gt p(S)=t\gt g[t]\oplus\gt p_{\Pi}(S)$. 
Suppose that $\alpha_0\in S$ and
$|S|\le r$. Let $\gt n$ be the nilpotent radical of $\gt p_{\Pi}(S)$.
Then $\gt p_{-1}=\gt z(\gt n) t^{-1}$, where $\gt z(\gt n)$  is the centre of $\gt n$.  
Hence $[\gt p_{-1},\gt p_{-1}]=0$ and $\gt p(S)=t\gt g[t]\oplus\gt p_{\Pi}(S)\oplus\gt p_{-1}$. 

Having two standard parabolics $\gt p$ and $\gt r$, we define an {\it affine seaweed subalgebra}  
$\gt q=\gt p\cap\gt r^-$.  
An interesting situation occurs if $\gt q$ is finite-dimensional. 
This is the case if and only if both $\gt p,\gt r$ are smaller than $\gt g[t,t^{-1}]$, cf.~\cite[Sect.\,1.4]{jos2}. 

In a slightly more general setting of affine Kac--Moody algebras, a combinatorial formula for the index of 
a finite-dimensional affine seaweed subalgebra is obtained in \cite{jos2}. We are interested in very particular examples 
of contact seaweeds and for them the general formula is not needed.  

Suppose that $\g$ is of type {\sf A}$_r$. Then the extended Dynkin diagram of $\gt{sl}_{r+1}$  
 is a cycle. We consider $\gt q=\gt p\cap \gt r^-$, where 
$\gt p=\gt p(S)$, $\gt r=\gt p(T)$, and $S,T\subset\tilde\Pi$ are proper subsets. 
Therefore   it is safe to 
assume that $\alpha_0\not\in T$. If also  $\alpha_0\not\in S$, then $\gt q$ is a seaweed 
subalgebra of $\g$. Such subalgebras have been discussed already. 

\subsection{Intersections of two maximal parabolics.} 
We will treat one instance, namely, where $\g=\gt{sl}_{r+1}$, $T=\Pi$, and
$|S|=\tilde\Pi\setminus\{\ap_i\}$ with $1\le i\le r$. 
Here 
$\gt q(S,T)$ is a semi-direct product 
\[
\gt q(S,T)\cong \gt{s}(\gt{gl}_a\oplus\gt{gl}_b)\ltimes 2\bbk^a{\otimes}\bbk^b, 
\]
where $a+b=r+1$ and $a=i$, the subspace $2\bbk^a{\otimes}\bbk^b$ is an Abelian ideal,  
the $1$-dimensional centre of the Levi subalgebra $\gt{s}(\gt{gl}_a\oplus\gt{gl}_b)\subset\gt{sl}_{r+1}$ acts on both copies of 
$\bbk^a{\otimes}\bbk^b$ with one and the same non-trivial character. 

First we compute the index of $\q$. Here a special case of \eqref{semi-ind} will be frequently used. For  a semi-direct product 
$H=L\ltimes\exp(V)$ such that $\exp(V)$ is Abelian and $L$ acts on $V^*$ with an open orbit $L\gamma$, it reads
\begin{equation} \label{r}
\ind\Lie H=\ind \Lie L_\gamma\,.
\end{equation}

It is convenient to extend $\q$ by a $1$-dimensional central 
toral subalgebra and to consider another semi-direct product, $R$. 
Set 
\[
Q= Q(a,b)=(\GL_a\times\GL_b)\ltimes \exp(2 \bbk^a{\otimes}\bbk^b)
\] and also 
\[
R=R(a,b)=(\GL_a\times\GL_b)\ltimes \exp((\gt{gl}_a^{\sf ab}\oplus  W)\oplus V),
\]
where 
$[\gt{gl}_a^{\sf ab},\gt{gl}_a^{\sf ab}]=[W,W]=0$,
$W\cong V\cong \bbk^a{\otimes}\bbk^b$, and $[\gt{gl}_a^{\sf ab}, W]=V$. 
In both cases, $\GL_a$ and $\GL_b$ act on $\bbk^a$ and $\bbk^b$ in the natural way, also 
$\GL_a$ acts on $\gt{gl}_a$ via the adjoint representation. 
Note that $Q(a,b)=Q(b,a)$, but  
$R(a,b)\ne R(b,a)$, if $a\ne b$. 
Set $\gt q=\gt q(a,b)=\Lie Q(a,b)$ and 
$\gt r=\gt r(a,b)=\Lie R(a,b)$. Assume that $\q\ne 0$ and $\gt r\ne 0$. 

\begin{thm} \label{t-ind}
We have  $\ind\gt q=\gcd(2a,a+b)$ and $\ind\gt r=\gcd(2a,b)$. 
\end{thm}
\begin{proof}
We argue 
by induction on $a+b$ dealing with both statements simultaneously.  
If $a=0$, then $Q(0,b)=R(0,b)=\GL_b$ and $\ind\gt q=b=\gcd(0,b)$. 

If $b=0$, then $R(a,0)=\GL_a\ltimes\exp(\gt{gl}_a^{\sf ab})$ is a {\it Takiff}\/ Lie algebra (truncated current algebra) and $\ind\gt r=2a=\gcd(2a,0)$ \cite{takiff}. In particular, both statements hold, if $a+b = 1$. 
Now suppose that $ab>0$. 

Set $H=\GL_a\times\GL_b$.
Consider first $Q$. Assume that $a\le b$.
 If $2a\le b$, then $H$ acts on 
$(2 \bbk^a{\otimes}\bbk^b)^*$ with an open orbit, say $H\xi$. Furthermore, 
$$
H_\xi=(\GL_a\times\GL_{b-2a})\ltimes \exp(2\bbk^a{\otimes}\bbk^{b-2a})=Q(a,b-2a).
$$ 
By~\eqref{r} and induction, 
$\ind\gt q=\ind \Lie H_{\xi}=\gcd(2a,b-a)=\gcd(2a,a+b)$. 

Keep the assumption $a\le b$. We can decompose $Q$ as 
$Q=(H\ltimes \exp(V_1))\ltimes \exp(V_2)$ with $V_1\cong V_2\cong \bbk^a{\otimes}\bbk^b$. 
Then $L=H\ltimes \exp(V_1)$ acts on $V_2^*$ with an open orbit, say $L\gamma$. 
Therefore 
$\ind\gt q=\ind \Lie L_\gamma$ by \eqref{r}.  We have $L_\gamma=H_\gamma\ltimes\exp(V_1)$. By a direct calculation, 
$$
L_\gamma=\GL_a\times\GL_{b-a}\ltimes\exp((\gt{gl}_a^{\sf ab}\oplus \tilde W)\oplus \tilde V)
$$
with $[\gt{gl}_a^{\sf ab},\gt{gl}_a^{\sf ab}]=[\tilde W,\tilde W]=0$,  $[\gt{gl}_a^{\sf ab},\tilde W]=\tilde V$, and $\tilde W\cong\tilde V\cong\bbk^a{\otimes}\bbk^{b-a}$, i.e., 
$L_\gamma=R(a,b-a)$. In case $2a\le b$, we can conclude that
$$
\ind\gt r(a,b-a)=\ind\gt q(a,b)=\ind\gt q(a,b-2a)=\gcd(2a,b-a). 
$$  
Therefore it remains to perform the induction step for $\gt r(a,b)$ in the case, where  $2a>a+b$, i.e., for $a>b$. 

Consider now $\gt r=\gt r(a,b)$ with $a>b$. Let $\eta\in V^*$ be a generic point. We regard it is an element of 
$\gt r^*$. There is a more suitable decomposition of $\gt r$, namely,  $\gt r=\tilde{\gt l}\ltimes (W\oplus V)$, where 
$\tilde{\gt l}=\Lie \tilde L$ with $\tilde L=H\ltimes\exp(\gt{gl}_a^{\sf ab})$. Note that $[W\oplus V,W\oplus V]=0$. 
The inequality $a>b$ implies 
that $\ad^*(\gt{gl}_a^{\sf ab}){\cdot}\eta=W^*$ and that $\tilde L\eta$ is open in $(W\oplus V)^*$. 
Therefore $\ind\gt r=\ind\tilde{\gt l}_\eta$  by \eqref{r}. By a straightforward computation, 
$\tilde{\gt l}_\eta=\gt r(a-b,b)$.  
Since $b>0$, we have $(a-b)+b<a+b$ and
 $\ind\tilde{\gt l}_\eta=\gcd(2(a-b),b)=\gcd(2a,b)$ by the inductive  hypothesis.  
\end{proof}

If $\gcd(2a,a+b)=2$, then removing the centre  of $\gt q(a,b)$ we obtain an affine seaweed of index $1$. Now a natural question is whether it 
is contact or not. 

\subsection{Affine seaweeds of index $1$}  \label{aff1}
Recall that  $\gt q(a,b)$ is a central extension of a standard affine seaweed 
$\gt q(S,T)$. Set $\bar\q=\bar\q(a,b)=\q(S,T)$. By Theorem~\ref{t-ind},
$$
\ind\bar\q=\ind\gt q(a,b)-1=\gcd(2a,a+b)-1. 
$$ 
 If $\ind\bar\q=1$, then either $a\in 2\BZ$ and $a+b\in 2+4\BZ$ or 
both $a$ and $b$ are odd. 
Let $\bar R(a,b)$ be the  quotient of $R(a,b)$ by the central toral subgroup.  

The centre of $\bar\q(a,b)$ is trivial. Thereby $\bar\q(a,b)$ is strongly quasi-reductive if 
and only if it is quasi-reductive. If $b\ne 0$, then the  centre of $\bar{\gt r}(a,b)$ is trivial. 
In case of $\bar{R}(a,0)=\SL_a\ltimes\exp(\gt{gl}_a^{\sf ab})$ with $a\ne 0$, the centre is $1$-dimensional and consists of 
unipotent elements. The corresponding Lie algebra $\bar{\gt r}(a,b):=\Lie \bar R(a,b)$  
is not quasi-reductive for $a>1$ and is quasi-reductive, but not 
strongly quasi-reductive, if $a=1$. 
 
\begin{thm} \label{sw1}
If $a$ and $b$ are even and $\ind\bar\q(a,b)=\ind\bar{\gt r}(a,b)=1$, 
then $\bar\q(a,b)$ and $\bar{\gt r}(a,b)$ are strongly quasi-reductive and hence contact. 
\end{thm}
\begin{proof}
We argue 
by induction on $a+b$. The case $a+b=0$ does not take place. 
Suppose that  $a+b=2$. Then either $a$ or $b$ is equal to zero. Note that 
$\bar\q(0,2)=\bar\q(2,0)=\gt{sl}_2$ is reductive. 
Since $\ind\bar{\gt r}(a,b)=1$, the only possibility for $\bar{\gt r}(a,b)$ is $\bar{\gt r}(0,2)=\gt{sl}_2$.
Now we have the induction base and assume that $a+b>2$. 
This assumption implies that $ab>0$.  

{\it Induction step for $\bar Q$.} \ 
The group $\bar Q(a,b)$ is a semi-direct product of $L={\rm S}(\GL_a\times\GL_b)\ltimes\exp(\bbk^a{\otimes}\bbk^b)$ and 
$\exp(V)$ for $V=\bbk^a{\otimes}\bbk^b$. As was mentioned in the proof of Theorem~\ref{t-ind}, $L$ acts on $V^*$ with an open 
orbit $L\gamma$, where $L_\gamma=\bar R(a,b-a)$. 
Since $a>0$, we have $a+(b-a)<a+b$. 

Because $a$ is even, $\bar{\gt r}(a,b-a)$ is strongly quasi-reductive and hence so is 
$\bar\q(a,b)$ by Proposition~\ref{impl2}.

{\it Induction step for $\bar R$.} \ 
The group $\bar R(a,b)$ has two different decompositions into semi-direct products $L\ltimes\exp(V)$. 
Suppose $a\le b$. Then $L={\rm S}(\GL_a\times\GL_b)\ltimes\exp(W)$ and 
$V=\gt{gl}_a^{\sf ab}\oplus V_0$, where $V_0\cong W\cong\bbk^a{\otimes}\bbk^b$ and 
 $[W,\gt{gl}_a^{\sf ab}]=V_0$. 
If 
$a>b$, then 
$$
L={\rm S}(\GL_a\times\GL_b)\ltimes\exp(\gt{gl}_a^{\sf ab})
$$ 
and $V=W\oplus V_0$. 
In both cases, $L$ acts on $V^*$ with an open orbit and a generic point $\gamma\in V^*$ can be chosen 
in $V_0^*$. 
\begin{gather*}
  \text{ If $a\le b$, then } \ L_\gamma=\bar Q(a,b-a);\\
  \text{ if $a > b$, then } \ L_\gamma=\bar R(a-b,b). 
\end{gather*}
Since $a$ is even,  we conclude that $\gt r(a,b)$ is strongly quasi-reductive using Proposition~\ref{impl2} and the inductive hypothesis. 

Finally recall that each strongly quasi-reductive Lie algebra is quasi-reductive and each quasi-reductive Lie algebra of index $1$ is contact by~\eqref{qred}. 
\end{proof}

Keep the assumption $\ind\bar\q(a,b)=\ind\bar{\gt r}(a,b)=1$. 
If $a$ is odd, then neither $\bar\q(a,b)$ nor $\bar{\gt r}(a,b)$ is strongly quasi-reductive, since the above reductions
end with  $\bar{\gt r}(1,0)$. At each reduction step $\bar\q(a,b)\ard \bar{\gt r}(a,b-a)$,  
\,$\bar{\gt r}(a,b)\ard \bar\q(a,b-a)$ or $\bar{\gt r}(a,b)\ard \bar{\gt r}(a-b,b)$, the normaliser $\gt l_{<\gamma>}$ is Frobenius.
This makes inductive arguments difficult.  

\begin{ex}\label{notc}
The Lie algebra $\bar\q(1,1)=\bbk\ltimes 2\bbk$ is not contact. 
We can decompose any $\bar Q(1,b)$ with an odd $b$ as $L\ltimes\exp(V)$, where 
$L=\GL_b$ and $V=2\bbk^b$. Then $L_{<\gamma>}=Q(1,b-2)$ for $\gamma\in V^*$ belonging to the open $L$-orbit. 
The Lie algebra of this normaliser is of index $2$, see Theorem~\ref{t-ind}. Thereby we can use Theorem~\ref{impl} and conclude 
by induction on $b$ that $\bar\q(1,b)$ is never contact.   
\end{ex}

Consider the chain of reductions $\bar\q(1,3)\ard\bar{\gt r}(1,2)\ard\bar\q(1,1)$, where the first and the last algebras are not contact. 
By Theorem~\ref{impl}{\sf (i)},  $\bar{\gt r}(1,2)$ is contact. 

Consider another chain of reductions
\[
\bar\q(3,5)\ard\bar{\gt r}(3,2)\ard\bar{\gt r}(1,2)\ard\bar\q(1,1).
\] 
Since $\bar{\gt r}(1,2)$ is contact, the previous item, $\bar{\gt r}(3,2)$, may be contact or not. In order to decide this, one 
can use Proposition~\ref{impl-ii}. 
If  $\bar{\gt r}(3,2)$ is contact, then 
the situation with $\bar\q(3,5)$ is the same. We suspect that $\bar\q(a,b)$ with odd  $a$ and $b$ is never contact,
but leave the question open for the moment. 


\begin{thebibliography}{15}

\bibitem{amm}
{\sc K.~Ammari}, 
Stabilit{\'e} des sous-alg\`ebres biparaboliques des alg\`ebres de Lie simples, 
{\it J. Lie Theory}, {\bf 32}, no.\,1 (2022), 239--260.

\bibitem{bgr} 
{\sc W.~Borho, P.~Gabriel}, and {\sc R.~Rentschler}, 
``Primideale in einh\"ullenden aufl\"osbarer Lie-Algebren'', 
Lecture Notes in Math., Bd. {\bf 357}, Springer-Verlag, 1973.

\bibitem{MZ}
{\sc J.-Y.~Charbonnel} and {\sc A.~Moreau},
The symmetric invariants of centralizers and Slodowy grading, 
{\it Math. Z.}, {\bf 282}\,(2016), no.~1-2, 273--339.

\bibitem{coll}
{\sc V.\,E.~Coll Jr.}  and {\sc N.~Russoniello}, 
Classification of contact seaweeds,
{\it J. Algebra}, {\bf 659}\,(2024), 811--817. 

\bibitem{derK}
{\sc V.~Dergachev} and {\sc A.A.~Kirillov}, 
Index of Lie algebras of seaweed type, 
{\it J. Lie Theory}, {\bf 10}\,(2000), 331--343.

\bibitem{dia}
{\sc Andr{\'e} Diatta}, 
Left invariant contact structures on Lie groups,
{\it Diff. Geometry and its Applications}, {\bf 26}\,(2008), 544--552. 

\bibitem{duflo}
 {\sc M.~Duflo, M.\,S.~Khalgui},  and {\sc P.~Torasso}, 
 Alg\`ebres de Lie quasi-r{\'e}ductives, {\it Transform. Groups}, {\bf 17}\,(2012),
417--470.

\bibitem{El-fr}
{\sc A.\,G.~Elashvili},
Frobenius Lie algebras, {\it Funct. Anal. Appl.}, {\bf 16}, no.\,4 (1982), 326--328. 

\bibitem{gei}
{\sc H.~Geiges}, 
An Introduction to Contact Topology, 
Cambridge Studies in Advanced Mathematics, {\bf 109}, 
Cambridge University Press 2008. 

\bibitem{gro}
{\sc M.~Gromov},
Stable maps of foliations in manifolds, {\it Izv. Akad. Nauk SSSR}, {\bf 33}\,(1969), 707--734.

\bibitem{jos}
{\sc A.~Joseph},  
On semi-invariants and index for biparabolic (seaweed) algebras, I, 
{\it J. Algebra}, {\bf 305}\,(2006), 487--515.

\bibitem{jos2}
{\sc A.~Joseph},
On semi-invariants and index for biparabolic (seaweed) algebras II, {\it J. Algebra}, 
{\bf 312}\,(2007), 158--193. 

\bibitem{js}
{\sc A.~Joseph} and {\sc D.~Shafrir},
Polynomiality of invariants, unimodularity and adapted pairs, 
{\it Transform. Groups}, {\bf 15}, no.\,4 (2010), 851--882.

\bibitem{MRS}
{\sc A.~Moreau} and {\sc O.~Yakimova}, 
Coadjoint orbits of reductive type of parabolic and seaweed Lie subalgebras, 
{\it Int. Math. Res. Notices}, {\bf 2012}, no.\,19 (2012), 4475--4519.

\bibitem{Ooms}
{\sc A.\,I.~Ooms}, 
On Frobenius Lie algebras, {\it Comm. Algebra}, {\bf 8},  no.\,1 (1980), 13--52.

\bibitem{fonya-et}
{\sc A.~Ooms} and {\sc M.~Van~den~Bergh}, 
A degree inequality for Lie algebras with a regular Poisson semi-center, 
{\it J. Algebra}, {\bf 323}, no.\,2 (2010), 305--322.

\bibitem{Dima}  {\sc D.~Panyushev}, 
Inductive formulas for the index of seaweed Lie algebras, {\it Moscow Math. J.}, {\bf 1}\,(2001), 221--241.

\bibitem{Dima03-b}  {\sc D.~Panyushev}, 
An extension of Ra\"{i}s'  theorem and seaweed subalgebras of simple Lie algebras,
{\it Ann. Inst. Fourier} (Grenoble), {\bf 55}\,(2005), no. 3, 693--715.

\bibitem{p09}
{\sc D.~Panyushev}, 
Periodic automorphisms of Takiff algebras, contractions, and $\theta$-groups,
{\it Transformation Groups}, {\bf 14}, no.\,2 (2009), 463--482.

\bibitem{mC} {\sc D.~Panyushev} and {\sc O.~Yakimova}, 
On seaweed subalgebras and meander graphs in type {\sf C}, 
{\it Pacific J. Math.}, {\bf 285}, no.\,2~(2016), 485--499. 

\bibitem{kot-T}
{\sc D.~Panyushev} and {\sc  O.~Yakimova}, 
Takiff algebras with polynomial rings of symmetric invariants, 
{\it Transformation Groups}, {\bf 25}\,(2020),  609--624. 

\bibitem{r}
{\sc M.~Ra{\"i}s}, 
L'indice des produits semi-directs $E\times_{\rho}\gt g$, \ 
{\it C. R. Acad. Sci. Paris Ser. A}, {\bf 287}\,(1978),
195--197.

\bibitem{sal}
{\sc G.~Salgado-Gonz{\'a}lez}, 
Invariants of contact Lie algebras, 
{\it J. Geometry and Physics}, {\bf 144}\,(2019), 388--396. 

\bibitem{SK} {\sc M.~Sato} and {\sc T.~Kimura}, 
A classification of irreducible prehomogeneous vector spaces and their relative invariants, 
{\it Nagoya Math. J.},  {\bf 65}\,(1977), 1--155.

\bibitem{spr} 
{\sc T.\,A.~Springer},  Aktionen reduktiver Gruppen auf Variet\"aten,
in: ``Algebraische Transformationsgruppen und Invariantentheorie", DMV-Seminar, Bd. {\bf 13}, 
Basel--Boston--Berlin: Birkh\"auser 1989, S.\,3--39.

\bibitem{takiff} 
{\sc S.\,J. Takiff},   
Rings of invariant polynomials for a class of Lie algebras, 
{\it Trans. Amer. Math. Soc.} {\bf 160}\,(1971), 249--262.

\bibitem{TYu} 
{\sc P.~Tauvel} and {\sc R.\,W.\,T.~Yu},
Indice et formes lin{\'e}aires stables dans les alg\`ebres de Lie, 
{\it J. Algebra}, {\bf 273}\,(2004), 507--516.

\bibitem{pisa}
{\sc O.~Yakimova}, 
Some semi-direct products  with free algebras of symmetric invariants, 
F. Callegaro et al. (eds.), {\it Perspectives in Lie Theory},
Springer INdAM Series {\bf 19}\,(2017), 267--279. 


\end{thebibliography}
\end{document}